\documentclass{amsart}

\usepackage{txfonts}
\usepackage[all]{xy}

\newtheorem{theorem}{Theorem}[section]

\newtheorem{proposition}[theorem]{Proposition}
\newtheorem{corollary}[theorem]{Corollary}

\theoremstyle{definition}
\newtheorem{definition}[theorem]{Definition}

\theoremstyle{remark}
\newtheorem{remark}[theorem]{Remark}

\numberwithin{equation}{section}

\begin{document}

\title{Characterizations of functions in wandering subspaces of the Bergman Shift via the Hardy space of the Bidisc}

\author{Sun Shunhua}
\address{Department of Mathematics, Jiaxing University, Jiaxing, Zhejiang, China 314001}
\email{shsun@mail.zjxu.edu.cn}

\author{Xu Anjian$^{\ast}$}
\address{School of Mathematical Sciences,
Chongqing University of Technology, Chongqing, China 400054}
\email{xuaj@cqut.edu.cn}
\thanks{The second author was Supported by NSFC (11871127).}
\thanks{* Correponding author}

\subjclass[2000]{Primary 47B35; Secondary 47B32}

\date{March 6, 2022 and, in revised form, June 22, 2022.}


\keywords{Invariant Subspace, Wandering Subspace, Bergman Shift}

\begin{abstract}
Let $\mathcal{W}$ be the corresponding wandering subspace of an invariant subspace of the Bergman shift. By identifying the Bergman space with $H^2(\mathbb{D}^2)\ominus[z-w]$, a sufficient and necessary conditions of a closed subspace of $H^2(\mathbb{D}^2)\ominus[z-w]$ to be a wandering subspace of an invariant subspace is given also, and a functional charaterization and a coefficient characterization for a function in a wandering subspace are given. As a byproduct, we proved that for two invariant subspaces $\mathcal{M}$, $\mathcal{N}$ with $\mathcal{M}\supsetneq\mathcal{N}$ and $dim(\mathcal{N}\ominus B\mathcal{N})<\infty$ $dim(\mathcal{M}\ominus B\mathcal{M})=\infty$, then there is an invariant subspace $\mathcal{L}$ such that $\mathcal{M}\supsetneq\mathcal{L}\supsetneq\mathcal{N}$. Finally, we define an operator from one wandering subspace to another, and get a decomposition theorem for such an operator which is related to the universal property of the Bergman shift.
\end{abstract}

\maketitle

\section{Introduction}
Let $\mathbb{D}$ be the open unit disk in the complex plane $\mathbb{C}$, $\mathbb{T}$ is the unit circle which is the boundary of $\mathbb{D}$. $H^{2}(\mathbb{T})$ denotes the Hardy space on $\mathbb{T}$. Let $dA$ denote Lebesgue area measure on $\mathbb{D}$, normalized so that the measure of $\mathbb{D}$ equals $1$. The Bergman space $L_{a}^{2}(\mathbb{D})$ is
the Hilbert space consisting of the analytic functions on $\mathbb{D}$ that are also in the space $L^2(\mathbb{D},dA)$ of square integrable functions on $\mathbb{D}$. The operator of multiplication by the coordinate function $z$ on the Bergman space is called the Bergman shift. The torus $\mathbb{T}^{2}$ is the Cartesian product $\mathbb{T}\times\mathbb{T}$. The Hardy space $H^2(\mathbb{T}^{2})$ over the bidisk is $H^2(\mathbb{T})\otimes H^2(\mathbb{T})$.

For an operator $T$ on a Hilbert space $H$, the structure of an operator is an important research subject in operator theory\cite{BCHL}. A subspace $M$ of $H$ is called an invariant subspace of $T$ if $TM\subset M$. The famous invariant subspace problem is one of the central problems in operator theory. If $M$ is an invariant subspace of $T$, $M\ominus TM$ is called the wandering subspace of $T$ on $M$. The famous Beurling theorem \cite{B1949} says that all invariant subspaces of the unilateral shift which can be identified as multiplication by $z$ on the Hardy space are generated by their wandering subspaces with dimension 1. The study of the Bergman shift attract great attentions since its universal property\cite{DS2004}. The remarkable Beurling type theorem of Aleman, Richter and Sundberg \cite{ARS1996} implies every invariant subspace of the Bergman shift is generated by its wandering subspace with possible dimension from 1 to $\infty$.

For each integer $n\geq 0$, let
\[p_{n}(z,w)=\sum_{i=0}^{n}z^{i}w^{n-i}.\]
Rudin \cite{R1969} defined $\mathcal{H}$ to be the subspace of $H^{2}(\mathbb{T}^{2})$ spanned by functions $\{p_{n}\}_{n=0}^{\infty}$. Thus every function in $\mathcal{H}$ is symmetric with respect to $z$ and $w$. Let $[z-w]$ denote the closure of $(z-w)H^{2}(\mathbb{T}^{2})$ in $H^{2}(\mathbb{T}^{2})$. As every function in $[z-w]$ is orthogonal to each $p_{n}$, it is well-known and easy to see that
\[H^{2}(\mathbb{T}^{2})=\mathcal{H}\oplus[z-w].\]
For a subspace $\mathcal{M}$ of $L^{2}(\mathbb{T}^{2})$ and let $P_{\mathcal{M}}$ denote the orthogonal projection from $L^{2}(\mathbb{T}^{2})$ onto $\mathcal{M}$. The Toeplitz operator on $H^{2}(\mathbb{T}^{2})$ with symbol $f$ in $L^{\infty}(\mathbb{T}^{2})$ is defined by
\[T_{f}(h)=P_{H^{2}(\mathbb{T}^{2})}(fh),\]
for $h$ in $H^{2}(\mathbb{T}^{2})$. It is not difficult to see that $T_{z}$ and $T_{w}$ are a pair of doubly commuting pure isometries on $H^{2}(\mathbb{T}^{2})$. It is easy to check that
\[P_{\mathcal{H}}T_{z}|_{\mathcal{H}}=P_{\mathcal{H}}T_{w}|_{\mathcal{H}}.\]
Let $B$ denote the operator above. It was shown explicitly in \cite{SY1989} and implicitly in \cite{DP1989} that $B$ is unitarily equivalent to the Bergman shift, the multiplication operator by the coordinate function $z$ on the Bergman space $L_{a}^{2}$ via the following unitary operator $U:L_{a}^{2}(\mathbb{D})\rightarrow \mathcal{H}$,
\[Uz^{n}=\frac{p_{n}(z,w)}{n+1}.\]
So the Bergman shift is lifted up as the compression of an isometry on a nice subspace $\mathcal{H}$ of $H^{2}$. In the rest of the paper we identify the Bergman shift with the operator $B$, and so an invariant subspace of the Bergman shift can be identified as a subspace of $\mathcal{H}$ which is invariant under $B$.

In last twenty years, it become an important tool to use the theory of multivariable operators and functions to study a single operator and the functions of one variable in the study of the Bergman shift. The idea is to lift the Bergman shift up as the compression of a commuting pair of isometries on a subspace of the Hardy space of the bidisk which was given in Rudin's book \cite{R1969}, used in studying the Hilbert modules by R. Douglas and V Paulsen \cite{DP1989}, operator theory in the Hardy space over the bidisk by R. Douglas and R. Yang \cite{DY2000,Y2001,Y2002}, the lattice of the invariant subspaces of the Bergman shift by S. Richter, the reducing subspaces of the Bergman shift by Zheng, Guo, Zhong, and the authors \cite{GSZZ2009,SZZ2008,SZZ2010}, the Beurling type theorem of the Bergman shift by Zheng and the fisrt author \cite{SZ2009} etc. In this paper, the idea is used to study the characterization of a function in the wandering subspace of any invariant subspace of the Bergman shift.

\section{Charaterizations of a wandering subspace}
For a function in $\mathcal{H}$, we have the following seires reprsentation.
\begin{proposition}
	For any $q(z,w)\in\mathcal{H}$, then
	\[q(z,w)=\sum_{j=0}^{\infty}z^{j}T_{w}^{\ast j}q(0,w)\]
\end{proposition}
\begin{proof}
	The identity follows from 
	\begin{eqnarray}
		&q(z,w)&=\sum_{n=0}^{\infty}a_{n}\frac{p_{n}(z,w)}{\sqrt{n+1}}=\sum_{n=0}^{\infty}a_{n}\frac{1}{\sqrt{n+1}}\sum_{j=0}^{n}z^{j}w^{n-j}\nonumber\\
		&&=\sum_{j=0}^{\infty}z^{j}\sum_{n=j}^{\infty}a_{n}\frac{1}{\sqrt{n+1}}w^{n-j},\nonumber
	\end{eqnarray}
	since $q(z,w)\in\mathcal{H}$, and 
	\begin{eqnarray}
		&T_{w}^{\ast j}q(0,w)&=T_{w}^{\ast j}\sum_{n=0}^{\infty}a_{n}\frac{p_{n}(0,w)}{\sqrt{n+1}}=T_{w}^{\ast j}\sum_{n=0}^{\infty}a_{n}\frac{w^{n}}{\sqrt{n+1}}\nonumber\\
		&&=\sum_{n=j}^{\infty}a_{n}\frac{1}{\sqrt{n+1}}w^{n-j}.\nonumber
	\end{eqnarray}
\end{proof}

By Proposition 2.1, we can get following characterization of a function in the wandering subspace for an invariant subspace.
\begin{theorem}
	If $\mathcal{M}$ is an invariant subspace of the Bergman shift $B$, and $q(z,w)$ is a function in the wandering subspace $\mathcal{M}\ominus B\mathcal{M}$ of $\mathcal{M}$, then
	\[\sum_{j=0}^{\infty}|T_{w}^{\ast j}q(0,w)|^{2}=constant, \ \forall w\in\partial\mathbb{D}.\eqno{(1)}\] 
\end{theorem}
\begin{proof}
	Since $q(z,w)\in\mathcal{M}\ominus B\mathcal{M}$ and $\mathcal{M}\perp[z-w]$, for any $k\neq0$, we have
	\begin{eqnarray}
		&0&=\langle z^{k}q(z,w),q(z,w)\rangle_{H^{2}(\mathbb{T}^{2})}\nonumber\\
		&&=\langle T_{w}^{k}q(z,w),q(z,w)\rangle_{H^{2}(\mathbb{T}^{2})}\nonumber\\
		&&=\sum_{j=0}^{\infty}\langle w^{k}T_{w}^{\ast j}q(0,w),T_{w}^{\ast j}q(0,w)\rangle_{H^{2}(\mathbb{T})}\nonumber\\
		&&=\frac{1}{2\pi}\int_{0}^{2\pi}e^{ik\theta}\sum_{j=0}^{\infty}T_{w}^{\ast j}q(0,w)\overline{T_{w}^{\ast j}q(0,w)}d\theta\nonumber\\
		&&=\frac{1}{2\pi}\int_{0}^{2\pi}e^{ik\theta}\sum_{j=0}^{\infty}|T_{w}^{\ast j}q(0,w)|^{2}d\theta\nonumber
	\end{eqnarray}
	by the proposition. It means that $\sum_{j=0}^{\infty}|T_{w}^{\ast j}q(0,w)|^{2}$ is orthogonal to $w^{k}$ for any $k\neq0$, so it must be a constant.
\end{proof}
\indent
One the other side, the converse of the above theorem is true also.
\begin{theorem}
	If $\mathcal{W}$ is a closed subspace of $\mathcal{H}$, and every function $q(z,w)\in \mathcal{W}$ satisfies
	\[\sum_{j=0}^{\infty}|T_{w}^{\ast j}q(0,w)|^{2}=constant, \ \forall w\in\partial\mathbb{D}.\]
	Then there is a minimal invariant subspace $\mathcal{M}$ of $B$ such that
	\[\mathcal{M}\ominus B\mathcal{M}=\mathcal{W}.\]
\end{theorem}
\begin{proof}
	In fact, let $\mathcal{M}=Span{B^{j}\mathcal{W}:j=0,1,2,\cdots}$, then $\mathcal{M}$ is an invariant subspace and it is clear that
	\[\mathcal{W}\subseteq\mathcal{M}\ominus B\mathcal{M}.\]
	Now if $q_{0}\in(\mathcal{M}\ominus B\mathcal{M})\ominus\mathcal{W}$, by definition of $\mathcal{M}$, there exist polynomials $p_{n}$ such that $\sum\limits_{j=1}^{N}p_{j}(B)q_{n}\rightarrow q_{0}$, where $q_{n}\in\mathcal{W}$. So
	\begin{eqnarray}
		&\|q_{0}\|^{2}_{H^{2}(\mathbb{T}^{2})}&=\lim_{N\rightarrow\infty}\langle\sum_{j=0}^{N}p_{n}(B)q_{n},q_{0}\rangle\nonumber\\
		&&=\lim_{N\rightarrow\infty}\sum_{j=0}^{N}\langle q_{n},(p_{n}(B))^{\ast}q_{0}\rangle\nonumber\\
		&&=0\nonumber
	\end{eqnarray}
	since $(p_{n}(B))^{\ast}q_{0}\perp\mathcal{M}$ and $q_{n}\perp q_{0}$.
\end{proof}
\indent
It is easy to get a necessary and sufficient conditions of a closed subspace in $\mathcal{H}$ to be the wandering subspace of an invariant subspace $\mathcal{M}$.
\begin{corollary}
	An closed subspace $\mathcal{W}$ of $\mathcal{H}$ is the wandering subspace of an invariant subspace $\mathcal{M}$ if and only if every function $q(z,w)$ in $\mathcal{W}$ satisfies
	\[\sum_{j=0}^{\infty}|T_{w}^{\ast j}q(0,w)|^{2}=constant, \ \forall w\in\partial\mathbb{D}.\]
\end{corollary}
\begin{remark}
	It is easy to check that if $q_{1},q_{2}$ satisfy $(1)$, then every function in $Span\{q_{1},q_{2}\}$ satisfies $(1)$ if and only if 
	\[\sum_{j=0}^{\infty}T_{w}^{\ast j}q_{1}(0,w)\overline{T_{w}^{\ast j}q_{2}(0,w)}=constant, \ \forall w\in\partial\mathbb{D}.\]
\end{remark}

The famous invariant subspace problem is equivalent to the universal property of the Bergman shift, i.e., for two invariant subspaces $\mathcal{M}$, $\mathcal{N}$ with $\mathcal{M}\supsetneq\mathcal{N}$ and $dim(\mathcal{N}\ominus B\mathcal{N})=dim(\mathcal{M}\ominus B\mathcal{M})=\infty$, is there an invariant subspace $\mathcal{L}$ such that $\mathcal{M}\supsetneq\mathcal{L}\supsetneq\mathcal{N}$? And so it is natural to ask for two invariant subspaces $\mathcal{M}$ $\mathcal{N}$ with $\mathcal{M}\supsetneq\mathcal{N}$, is there an invariant subspace $\mathcal{L}$ such that $\mathcal{M}\supsetneq\mathcal{L}\supsetneq\mathcal{N}$? We get the following interesting theorem in a special case.
 
\begin{theorem}
	For two invariant subspaces $\mathcal{M}$ and $\mathcal{N}$ of $B$, if $\mathcal{M}\supsetneq\mathcal{N}$ and $\mathcal{N}\ominus B\mathcal{N}$ is finite dimensional, and $\dim(\mathcal{M}\ominus B\mathcal{M})=\infty$. Then there is an invariant subspace $\mathcal{L}$ of $B$ such that 
	\[\mathcal{N}\subsetneq\mathcal{L}\subsetneq\mathcal{M}.\]
\end{theorem}
\begin{proof}
	Assume that $dim(\mathcal{N}\ominus B\mathcal{N})=K$ and let $\{q_{j}\}_{j=1}^{K}$ be an orthonormal basis of $\mathcal{N}\ominus B\mathcal{N}$. We can choose an function $\widehat{q}$ with norm 1 in $\mathcal{M}\ominus B\mathcal{M}$ which is orthogonal to $q_{j}$ for $1\leq j\leq K$ since $\dim(\mathcal{M}\ominus B\mathcal{M})=\infty$. 
	Let $\mathcal{W}$ be the closed linear span of $\{q_{j}\}_{j=1}^{K}$ and $\widehat{q}$. 
	
	Next we show that every function $q$ in $\mathcal{W}$ satisfying
	\[\sum_{j=0}^{\infty}|T_{w}^{\ast j}q(0,w)|^{2}=constant, \ \forall w\in\partial\mathbb{D}.\]
	For simplicity and whitout loss of generality, we prove the statement for the case $K=1$. In this case $q=\lambda q_{1}+\mu\widehat{q}$ where $\lambda,\mu$ are two constants, and so 
	\[\begin{array}{ll}
		\sum\limits_{j=0}^{\infty}|T_{w}^{\ast j}q(0,w)|^{2}&=\sum\limits_{j=0}^{\infty}|\lambda T_{w}^{\ast j}q_{1}(0,w)+\mu T_{w}^{\ast j}\widehat{q}(0,w)|^{2}\\
		&=|\lambda|^2\sum\limits_{j=0}^{\infty} |T_{w}^{\ast j}q_{1}(0,w)|^{2}+|\mu|^{2}\sum\limits_{j=0}^{\infty} |T_{w}^{\ast j}\widehat{q}(0,w)|^{2}\\
		&\hspace{20pt}+2Re\lambda\overline{\mu}\sum\limits_{j=0}^{\infty}T_{w}^{\ast j}q_{1}(0,w)\overline{T_{w}^{\ast j}\widehat{q}(0,w)}\\
		&=constant,
	\end{array}\]
    by Theorem 2.2 and Remark 2.5 since $q_{1}$ and $\widehat{q}$ are in wandering subspaces $\mathcal{N}\ominus B\mathcal{N}$ and $\mathcal{M}\ominus B\mathcal{M}$ respectively.
	
	Then there is a minimal inariant subspace $\mathcal{L}$ of $B$ such that 
	\[\mathcal{L}\ominus B\mathcal{L}=\mathcal{W},\]
	by Theorem 2.3.	It is clear that $\mathcal{N}\subsetneq\mathcal{L}\subsetneq\mathcal{M}$, and so the theorem follows.
\end{proof}

\begin{remark}
	If on the Dirichlet space $\mathcal{D}(\mathbb{D})$, we define the following inner product
	\[\langle f,g\rangle_{\mathcal{D}}=\frac{1}{\pi}\iint_{\mathbb{D}}(wf(w))'\overline{(wg(w))'}dA(w)\]
	for $f,g\in\mathcal{D}(\mathbb{D})$. Then  $q_{f}=\sum\limits_{j=0}^{\infty}z^{j}T_{w}^{\ast j}f$ and $q_{g}=\sum\limits_{j=0}^{\infty}z^{j}T_{w}^{\ast j}g$ are in $\mathcal{H}$. Moreover, if $f(w)=\sum\limits_{k=0}f_{k}w^{k}$ and $g(w)=\sum\limits_{k=0}g_{k}w^{k}$, then we have
	\begin{eqnarray}
		&\langle q_{f},q_{g}\rangle_{H^{2}(\mathbb{T}^{2})}&=\sum_{j=0}^{\infty}\langle T_{w}^{\ast j}f,T_{w}^{\ast j}g\rangle_{H^{2}(\mathbb{T})}=\sum_{j=0}^{\infty}\langle\sum_{k=j}^{\infty}f_{k}w^{k-j},\sum_{m=j}^{\infty}g_{m}w^{m-j}\rangle\nonumber\\
		&&=\sum_{j=0}^{\infty}\sum_{k=j}^{\infty}f_{k}\overline{g_{k}}=\sum_{k=0}^{\infty}\sum_{j=0}^{k}f_{k}\overline{g_{k}}=\sum_{k=0}^{\infty}(k+1)f_{k}\overline{g_{k}}\nonumber\\
		&&=\langle f,g\rangle_{\mathcal{D}}.\nonumber
	\end{eqnarray}
\end{remark}

\begin{theorem}
	If a closed subspace $\mathcal{W}$ of $\mathcal{H}$ is the wandering subspace of an invariant subspace $\mathcal{M}$ if and only if every function $q(z,w)=\sum\limits_{j=0}^{\infty}z^{j}T_{w}^{\ast j}q(0,w)$ in $\mathcal{W}$ satisfies
	\[\left\{\begin{array}{ll}
		\sum\limits_{j=0}^{\infty}jq_{j}\overline{q_{j+k}}=0\\
		\sum\limits_{j=0}^{\infty}jq_{j+k}\overline{q_{j}}=0
	\end{array}\right.\ \forall k>0\eqno{(2)}\]
	where $q(0,w)=\sum\limits_{k=0}^{\infty}q_{k}w^{k}$.
\end{theorem}
\begin{proof}
	By Theorem 2.2, $\mathcal{W}$ is the wandering subspace of an invariant subspace $\mathcal{M}$ if and only if every function $q(z,w)$ in $\mathcal{W}$ satisfies
	\[\sum_{j=0}^{\infty}|T_{w}^{\ast j}q(0,w)|^{2}=constant, \ \forall w\in\partial\mathbb{D}.\]
	It is equivalent to that for any $k>0$,
	\begin{eqnarray}
		&&\frac{1}{2\pi}\int_{0}^{2\pi}w^{k}\sum_{j=0}^{\infty}|T_{w}^{\ast j}q(0,w)|^{2}d\theta=0\nonumber\\
		&\Leftrightarrow&0=\frac{1}{2\pi}\int_{0}^{2\pi}w^{k}\sum_{j=0}^{\infty}T_{w}^{\ast j}q(0,w)\overline{T_{w}^{\ast j}q(0,w)}d\theta\nonumber\\
		&&\hspace{8pt} =\sum_{j=0}^{\infty}\langle T_{w}^{\ast j}q(0,w), \overline{w^{k}}T_{w}^{\ast j}q(0,w)\rangle\nonumber\\
		&&\hspace{8pt} =\sum_{j=0}^{\infty}\langle T_{w}^{\ast j}q(0,w),T_{w}^{\ast j+k}q(0,w)\rangle\nonumber\\
		&&\hspace{8pt} =\langle q(0,w), T_{w}^{\ast k}q(0,w)\rangle\nonumber\\
		&&\hspace{8pt} =\sum\limits_{j=0}^{\infty}jq_{j}\overline{q_{j+k}}.\nonumber
	\end{eqnarray}
	For $k<0$, we get the second equation.
\end{proof}
\begin{remark}
	The theorem is a coefficient description of a function in wandering subspace.
\end{remark}
\begin{remark}
	Define a map from the Dirichlet space $\mathcal{D}$ to $\mathcal{H}$ by 
	\[f\mapsto\sum_{j=0}^{\infty}z^{j}T_{w}^{\ast j}f(w).\]
	The map is an isometry one-to-one by remark 2.5, the inverse map is $q(z,w)\mapsto q(0,w)$. Moreover, $T_{w}^{\ast}$ on $\mathcal{H}$ and $T_{w}^{\ast}$ on $\mathcal{D}$ satisfy the following identity
	\[\langle q_{f},T_{w}^{\ast}q_{g}\rangle_{H^{2}(\mathbb{T}^{2})}=\langle f,T_{w}^{\ast}g\rangle, \ \forall f,g\in\mathcal{D}\]
	It shows that $B^{\ast}|_{\mathcal{H}}\cong T_{w}^{\ast}|_{\mathcal{D}}$. It does not mean that $B|_{\mathcal{H}}$ is $P_{\mathcal{D}}T_{w}|_{\mathcal{D}}$. Indeed, $B|_{\mathcal{H}}$ is $(T_{w}^{\ast}|_{\mathcal{D}})^{\ast}$. $T_{w}^{\ast}|_{\mathcal{D}}$ is the adjoint of $T_{w}$ in the Dirichelet space, but $B^{\ast}|_{\mathcal{H}}$ is the adjoint of $P_{\mathcal{H}}T_{z}|_{\mathcal{H}}$ in the Hardy space $H^{2}(\mathbb{D}^{2})$.
\end{remark}

Wandering subspace of the Bergman shift has possible dimension from 1 to $\infty$, which is different to Hardy space on the unit disc. The following corollary give a possible construction of wandering subspace with diemsion gretaer than 1.

\begin{corollary}
	A closed subspace $\mathcal{W}$ of $\mathcal{H}$ is the wandering subspace of an invariant subspace of $B$ if and only if for  $q_{1}(z,w)=\sum\limits_{j=0}^{\infty}z^{j}T_{w}^{\ast j}q_{1}(0,w)$ and $q_{2}(z,w)=\sum\limits_{j=0}^{\infty}z^{j}T_{w}^{\ast j}q_{2}(0,w)$ in $\mathcal{W}$, the following identity holds
	\[\sum\limits_{j=0}^{\infty}T_{w}^{\ast j}q_{1}(0,w)\overline{T_{w}^{\ast j}q_{2}(0,w)}\in H^{1}(\mathbb{D})\eqno{(3)}\] 
\end{corollary}
\begin{proof}
	It is sufficient to show that (3) is holomorphic, which is equivalent to show that the Laurent coefficents with negative index are zero. As a fact, $\forall k>0$, we have
	\begin{eqnarray}
		&&\frac{1}{2\pi}\int_{0}^{2\pi}w^{k}\sum\limits_{j=0}^{\infty}T_{w}^{\ast j}q_{1}(0,w)\overline{T_{w}^{\ast j}q_{2}(0,w)}d\theta\nonumber\\
		&=&\frac{1}{2\pi}\sum\limits_{j=0}^{\infty}\int_{0}^{2\pi}T_{w}^{\ast j}q_{1}(0,w)\overline{\overline{w}^{k}T_{w}^{\ast j}q_{2}(0,w)}d\theta\nonumber\\
		&=&\sum\limits_{j=0}^{\infty}\langle T_{w}^{\ast j}q_{1}(0,w),\overline{w}^{k}T_{w}^{\ast j}q_{2}(0,w)\rangle_{L^{2}(\mathbb{T})}\nonumber\\
		&=&\sum\limits_{j=0}^{\infty}\langle T_{w}^{\ast j}q_{1}(0,w),T_{w}^{\ast j+k}q_{2}(0,w)\rangle_{L^{2}(\mathbb{T})}\nonumber\\
		&=&\langle q_{1}(z,w),T_{w}^{\ast k}q_{2}(z,w)\rangle_{H^{2}(\mathbb{T}^{2})}=0,\nonumber
	\end{eqnarray}
	since $q_{1}(z,w)$ is orthogonal to $T_{w}^{\ast k}q_{2}(z,w)$.
\end{proof}

\section{Wandering subspaces with infinite dimension}
\indent 
Now we turn to study the internal structure of a wandering subspace $\mathcal{W}$ for the case of $\dim\mathcal{W}=\infty$.
\begin{proposition}
	If $\mathcal{M}$ is an invariant subspace of $B$, $\dim(\mathcal{M}\ominus B\mathcal{M})=\infty$, and $\{q_{j}\}_{j=1}^{\infty}$ is an orthonrmal basis of $\mathcal{M}\ominus B\mathcal{M}$. Then
	\[\sum\limits_{j=0}^{\infty}T_{w}^{\ast j}q_{k}(0,w)\overline{T_{w}^{\ast j}q_{k'}(0,w)}=\left\{\begin{array}{ll}
	1&\text{ if }k=k'\\
	0&\text{ if }k\neq k'
	\end{array}\right.\ \forall w\in\mathbb{T}.\]
\end{proposition}
\begin{proof}
	For case $k=k'$, it is just Equation (1). For case $k\neq k'$, note that $q_{j}\in\mathcal{W}$, we have for any $m>0$
	\begin{eqnarray}
		&0&=\langle z^{m}q_{k},q_{k'}\rangle_{H^{2}(\mathbb{T}^{2})}=\langle q_{k},T_{z}^{\ast m}q_{k'}\rangle_{H^{2}(\mathbb{T}^{2})}\nonumber\\
		&&=\langle q_{k},T_{w}^{\ast m}q_{k'}\rangle_{H^{2}(\mathbb{T}^{2})}\nonumber\\
		&&=\sum_{j=0}^{\infty}\langle T_{w}^{\ast j}q_{k}(0,w),T_{w}^{\ast m+j}q_{k'}(0,w)\rangle_{H^{2}(\mathbb{T})}\nonumber\\
		&&=\frac{1}{2\pi}\int_{0}^{2\pi}w^{m}\sum_{j=0}^{\infty}T_{w}^{\ast j}q_{k}(0,w)\overline{T_{w}^{\ast j}q_{k'}(0,w)}d\theta.\nonumber
	\end{eqnarray}
	For $m<0$, it is similar to show that $\frac{1}{2\pi}\int_{0}^{2\pi}w^{m}\sum_{j=0}^{\infty}T_{w}^{\ast j}q_{k}(0,w)\overline{T_{w}^{\ast j}q_{k'}(0,w)}d\theta=0$. Finally, $m=0$, it is zero since $q_{k}\perp q_{k'}$. So we have proved that for any $m\in\mathbb{Z}$,
	\[\frac{1}{2\pi}\int_{0}^{2\pi}w^{m}\sum_{j=0}^{\infty}T_{w}^{\ast j}q_{k}(0,w)\overline{T_{w}^{\ast j}q_{k'}(0,w)}d\theta=0,\]
	and so 
	\[\sum_{j=0}^{\infty}T_{w}^{\ast j}q_{k}(0,w)\overline{T_{w}^{\ast j}q_{k'}(0,w)}=0,\ \forall w\in\partial\mathbb{D},k\neq k'.\]    
\end{proof}
\begin{definition}
	For any $q(z,w)=\sum\limits_{j=0}^{\infty}z^{j}T_{w}^{\ast j}q(0,w)$, define $L_w:\mathcal{W}\mapsto \ell^{2}(\mathbb{N},H^{2}(\mathbb{T}))$, $\forall w\in\partial\mathbb{D}$ as follows:
	\[L_{w}q(z,w)=(q(0,w),T_{w}^{\ast}q(0,w),\cdots,T_{w}^{\ast j}q(0,w),\cdots)\]
\end{definition}
\indent $L_{w}$ is defined on $\partial\mathbb{D}$ pointwise. And if $q_{1},q_{2}$ in a wandering subspace $\mathcal{W}$ and $q_{1}$ is orthogonal to $q_{2}$ in $H^{2}(\mathbb{T}^{2})$, then $\langle L_{w}q_{1},L_{w}q_{2}\rangle_{\ell^{2}(\mathbb{N})}=0$, $\forall w\in\partial\mathbb{D}$. $\mathcal{W}$ can be viewed as the subspace of $\ell^{2}(\mathbb{N},H^{2}(\mathbb{T}))$ under the map $L_{w}$.
\begin{definition}
	If $\mathcal{M}$ and $\mathcal{N}$ are two invariant subspace of $B$, $\mathcal{M}\supseteq\mathcal{N}$, $\dim(\mathcal{M}\ominus B\mathcal{M})=\dim(\mathcal{N}\ominus B\mathcal{N})$, and $\{q_{k}\}_{k=1}^{\infty}$ and $\{\widetilde{q_{k}}\}_{k=1}^{\infty}$ are orthonomal basis of $\mathcal{N}\ominus B\mathcal{N}$ and $\mathcal{M}\ominus B\mathcal{M}$ respectively. Define $T_{w}$ as
	\[T_{w}(L_{w}q_{k})=L_{w}\widetilde{q_{k}}, \ \forall k\geq1,\forall w\in\partial\mathbb{D}.\]
\end{definition}
\begin{proposition}
	The operator $T_{w}$ is an isometry from $\mathcal{W}_{\mathcal{N}}$ to $\mathcal{W}_{\mathcal{M}}$.
\end{proposition}
\begin{proof}
	Firstly, $T_{w}$ is defined for basis of $\mathcal{W}_{\mathcal{N}}$, it can be extended to $\mathcal{W}_{\mathcal{N}}$. $T_{w}$ is an isomtetry by Proposition 2.10.
\end{proof}
\begin{proposition}
	If $\mathcal{M}$ and $\mathcal{N}$ are two invariant subspaces of $B$, $\mathcal{M}\supseteq\mathcal{N}$ and $\dim\mathcal{W}_{\mathcal{N}}=\dim\mathcal{W}_{\mathcal{M}}=\infty$. Then for any unitary operator $U:\mathcal{W}_{\mathcal{N}}\rightarrow\mathcal{W}_{\mathcal{M}}$, there is a unitary operator $V:\ell^{2}(\mathcal{W}_{\mathcal{N}})\rightarrow\ell^{2}(\mathcal{W}_{\mathcal{M}})$ such that $T_{w}U=VT_{w}$, i.e. the following diagram is commutative.
	\[\xymatrix{
		\mathcal{W}_{\mathcal{N}}\ar[rr]^{U}\ar[d]_{T_{w}} & & \mathcal{W}_{\mathcal{M}}\ar[d]_{T_{w}}\\
		\ell^{2}(\mathcal{W}_{\mathcal{N}})\ar@{-->}[rr]^{V}& &\ell^{2}(\mathcal{W}_{\mathcal{M}})
	}\]
\end{proposition}
\begin{proof}
	It is clear.
\end{proof}
\indent
If $\mathcal{M}$ and $\mathcal{N}$ are two invariant subspaces of $B$, for convenience, $T_{w}^{(\mathcal{M},\mathcal{N})}$ is used to denote the isometry $T_{w}$ from $\mathcal{W}_{\mathcal{M}}$ to $\mathcal{W}_{\mathcal{N}}$.
\begin{theorem}
	If $\mathcal{M}$, $\mathcal{L}$ and $\mathcal{N}$ are three invariant subspaces of $B$, $\mathcal{M}\supseteq\mathcal{L}\supseteq\mathcal{N}$, and $\dim\mathcal{W}_{\mathcal{M}}=\dim\mathcal{W}_{\mathcal{L}}=\dim\mathcal{W}_{\mathcal{N}}$, then we have the following decomposition
	\[T_{w}^{(\mathcal{M},\mathcal{N})}=T_{w}^{(\mathcal{L},\mathcal{N})}T_{w}^{(\mathcal{M},\mathcal{L})}.\]
\end{theorem}
\begin{proof}
	Assume that $\{q_{k}^{\mathcal{M}}\}_{k=1}^{\infty}$, $\{q_{k}^{\mathcal{L}}\}_{k=1}^{\infty}$, $\{q_{k}^{\mathcal{N}}\}_{k=1}^{\infty}$ are an orthonormal basis of $\mathcal{W}_{\mathcal{M}}$, $\mathcal{W}_{\mathcal{L}}$ and $\mathcal{W}_{\mathcal{N}}$ respectively, by definition, we have
	\[\left\{\begin{array}{lll}
	T_{w}^{(\mathcal{M},\mathcal{N})}(L_{w}q_{k}^{\mathcal{M}})&=&L_{w}q_{k}^{\mathcal{N}}\\
	T_{w}^{(\mathcal{M},\mathcal{L})}(L_{w}q_{k}^{\mathcal{M}})&=&L_{w}q_{k}^{\mathcal{L}}\\
	T_{w}^{(\mathcal{L},\mathcal{N})}(L_{w}q_{k}^{\mathcal{L}})&=&L_{w}q_{k}^{\mathcal{N}}\\
	\end{array}\right.\ \forall k>1,\forall w\in\mathbb{T}.\]
	The decomposition follows from a direct computation. 
\end{proof}
\begin{remark}
	The theorem shows that for two invariant subspaces $\mathcal{M}$ and $\mathcal{N}$ with $\mathcal{M}\supseteq\mathcal{N}$, if there is an invariant subspace $\mathcal{L}$ such that $\mathcal{M}\supsetneq\mathcal{L}\supsetneq\mathcal{N}$, then the holomorphic isometric valued function $T_{w}^{(\mathcal{M},\mathcal{N})}$ has a decomposition. So by the universal property of the Bergman shift, the famous invariant subspace problem is equivalent to the decomposition problem of the holomorphic isometric valued function $T_{w}^{(\mathcal{M},\mathcal{N})}$.
\end{remark}

Acknowledgement: The authors thank Yueshi Qing for an observation of Theorem 2.6.

Conflict of Interest: The authors declare that they have no conflict of interest.

\bibliographystyle{amsplain}

\end{document}